\documentclass[11pt,a4paper]{amsart}

\usepackage{amssymb, amsthm, amsmath, amsfonts, mathrsfs, datetime, multirow}
\usepackage{indentfirst}
\usepackage{graphicx}
\usepackage{scalerel, cmap, wrapfig}
\usepackage{tikz-cd}
\usetikzlibrary{calc, 3d, patterns}
\tikzset{
	labl/.style={anchor=south, rotate=-90, inner sep=.6mm}
}
\usepackage{dsfont}
\usepackage{stmaryrd}
\usepackage{enumitem}
\usepackage[labelformat=empty]{caption}
\usepackage{color}\definecolor{darkblue}{rgb}{0,0.1,.5}
\usepackage[colorlinks=true,linkcolor=darkblue, urlcolor=darkblue, 
citecolor=darkblue]{hyperref}
\usepackage[left=1.5cm, right=1.5cm, top=3cm, bottom=3cm]{geometry}
\makeatletter
\newcommand*{\RN}[1]{\expandafter\@slowromancap\romannumeral #1@}
\makeatother

\newcommand\sbullet[1][.5]{\mathbin{\vcenter{\hbox{\scalebox{#1}{$\bullet$}}}}}

\def\C{\mathbb{C}}

\def\Z{\mathbb{Z}}

\def\Fp{\mathbb F_p}

\def\Ap{\mathfrak A_p}

\def\Ext{\mathop{\mathrm{Ext}}\nolimits}
\def\Cotor{\mathop{\mathrm{Cotor}}\nolimits}

\def\id{\mathds{1}}

\DeclareMathOperator{\im}{Im}

\def\ge{\geqslant}

\newcommand{\hr}[2][]{\hyperref[#2]{#1~\ref{#2}}}

\newtheorem{theorem}{Theorem}[section]
\newtheorem*{theorem*}{Theorem}
\newtheorem{lemma}[theorem]{Lemma}

\newtheorem{proposition}[theorem]{Proposition}
\newtheorem{corollary}[theorem]{Corollary}

\theoremstyle{definition}

\newtheorem{remark}[theorem]{Remark} 

\renewenvironment{proof}[1][{\itshape Proof}]{{\itshape #1. }}{\qed}

\title{On homology of the $MSU$ spectrum}

\author{Semyon Abramyan}
\address{AG Laboratory, HSE, 6 Usacheva str., Moscow, Russia, 119048}
\email{semyon.abramyan@gmail.com}
\thanks{This work is supported by the Russian Science Foundation under grant № 21-71-00049 and partially funded by the Simons Foundation.}

\begin{document}

\begin{abstract}
    We give a complete proof the Novikov isomorphism
    $\varOmega^{SU}\otimes \Z[\textstyle\frac12]\cong\mathbb Z[{\textstyle\frac12}][y_2,y_3,\ldots],\quad \deg y_i=2i$, where $\varOmega^{SU}$ is the $SU$-bordism ring. The proof uses the Adams spectral sequence and a description of the comodule structure of $H_{\sbullet}(MSU; \mathbb F_p)$ over the dual Steenrod algebra $\Ap^*$ with odd prime~$p$, which was also missing in the literature.
\end{abstract}


\maketitle

\tableofcontents

\section{Introduction}
The theory of bordism and cobordism was actively developed in 
the~1950--1960s. Most leading topologists of the time have 
contributed to this development. The idea of bordism was first explicitly 
formulated by Pontryagin~\cite{pont55} who related the theory of framed bordism 
to the stable homotopy groups of spheres using the concept of transversality.
Key results of bordism theory were obtained in the works of Rokhlin~\cite{rohl59}, Thom~\cite{thom54},
Novikov~\cite{novi60, novi62}, Wall~\cite{wall60}, Averbuch~\cite{aver59}, 
Milnor~\cite{miln60}, Atiyah~\cite{atiy61}.

Topologists have quickly realised the potential of the Adams spectral sequence~\cite{adam58} for calculations in the bordism theory. It culminated in the description of the complex (or unitary) bordism ring $\varOmega^U$ in the works of Milnor~\cite{miln60} and 
Novikov~\cite{novi60,novi62}. The ring $\varOmega^U$ was shown to be isomorphic 
to a graded integral polynomial ring $\Z[a_i\colon i\ge1]$ on infinitely many 
generators, with one generator in every even degree, $\deg a_i=2i$. This result 
has since found numerous applications in algebraic topology and beyond.

In Novikov's 1967 work \cite{novi67} a brand new approach to cobordism and stable
homotopy theory was proposed, based on the application of the Adams--Novikov
spectral sequence and formal group laws techniques. This approach was
further developed in the context of bordism of manifolds with
singularities in the works of Mironov \cite{miro75}, Botvinnik \cite{botv92} and Vershinin \cite{vers93}.
The Adams--Novikov spectral sequence has also become the main
computational tool for stable homotopy groups of spheres \cite{rave86}.

As an illustration of his approach, Novikov outlined a complete
description of the additive torsion and the multiplicative structure of
the SU-bordism ring $\varOmega^{SU}$, which provided a systematic view on
earlier geometric calculations with this ring. A modernised exposition of
this description is given in the survey paper~\cite{c-l-p19} of Chernykh, Limonchenko and
Panov; it includes the geometric results by Wall \cite{wall66},
Conner--Floyd~\cite{co-fl66m} and Stong~\cite{ston68}, the calculations with the
Adams-Novikov spectral sequence, and the details of the arguments missing
in Novikov's work \cite{novi67}. A full description of the SU-bordism ring
$\varOmega^{SU}$ relies substantially on the calculation of
$\varOmega^{SU}$ with 2 invereted, namely on proving the ring isomorphism
\[
	\varOmega^{SU}\otimes_\Z \Z[\textstyle\frac12]\cong\mathbb Z[{\textstyle\frac12}][y_2,y_3,\ldots],\quad \deg y_i=2i.
\]
This result first appeared in Novikov's work \cite{novi62} with only a sketch of
the proof, stating that it can be proved using Adams' spectral sequence in
a way similar to Novikov's calculation of the complex bordism ring
$\varOmega^U$. Although the result has been considered as known since the
1960s, its full proof has been missing in the literature, and also not
included in the survey~\cite{c-l-p19}.

\medskip

The main goal of this work is to give a complete proof of the isomorphism above
using the original methods of the Adams spectral sequence. It is included as \hr[Theorem]{pi-msu}. While filling in details in Novikov's sketch we faced technical problems that seemed to be unknown before. For example, the comodule structure of $H_{\sbullet}(MSU; \mathbb F_p)$ over the dual Steenrod algebra $\Ap^*$ with odd prime $p$ has not been satisfactory described in the literature. This calculation is one of the main results of the paper (\hr[Theorem]{MSU-coalgebra}).
We also included a description of the related Hurewicz homomorphism $\pi_{\sbullet}(MSU)\to H_{\sbullet}(MSU)$ and forgetful map $\pi_{\sbullet}(MSU)\to \pi_{\sbullet}(MU)$ in appropriate generators (with $2$ inverted) in \hr[Corollary]{comm_sq}, and a result on the divisibility of characteristic numbers of
$SU$-manifolds (\hr[Theorem]{polynomial-generators}).

\medskip

The structure of the paper is as follows. In \hr[Section]{preliminaries},
we fix notation and recall the necessary information on bordism, cohomology 
operations and the Adams spectral sequence.
In \hr[Section]{sunitary}, we describe $H_{\sbullet}(MSU; \Fp)$ as a 
$\Ap^*$-comodule. Then in \hr[Section]{novikov_thm} using the Adams spectral sequences we prove the 
isomorphism $\varOmega^{SU} \otimes \Z[\frac{1}{2}]
{}\cong{}
\Z[\frac{1}{2}][y_2, y_3, \ldots]$. Finally, in \hr[Section]{generators} we descibe the Hurewicz homomorphism and compute the Milnor genus $s_n(y_n)$.

\medskip

I wish to express gratitude to my supervisor Taras E. Panov for stating the problem, valuable advice and countless hours of discussions. I thank my advisor Alexander A.~Gaifullin for stimulating discussions and significant contribution to this work.

\section{Preliminaries}\label{preliminaries}
In this section we recall necessary facts and notation. To fully explore the 
topics listed below, we recommend the following
sources~\cite{swit75,ston68,co-fl64,mi-st74,mo-ta68,st-ep62,rave86}.

\medskip

Let $G$ denote either the unitary group $U$ or the special unitary group $SU$.
Denote by $MG$ corresponding Thom spectrum, i.e. the spectrum whose spaces are
Thom spaces of the universal vector $G(n)$-bundles.

By the fundamental theorem of Pontryagin and Thom, the homotopy groups of the Thom
spectrum $MG$ are isomorphic to the bordism ring of manifolds with a $G$-structure on
the stable normal bundle.

\medskip

The main technical tool for computations of $\pi_{\sbullet}(MG)$
will be the Adams spectral sequence. We first recall some basic facts
about Steenrod operations.

Let $X$ be a topological space and $k$ be a non-negative integer.
There are natural cohomology operations called \emph{Steenrod operations}
\[
Sq^k\colon H^n(X; \mathbb F_2) \to H^{n+k}(X; \mathbb F_2),
\]
and for an odd prime~$p$
\[
P^k\colon H^n(X; \Fp) \to H^{n + 2n(p-1)}(X; \Fp).
\]

These operations are defined by the following properties.
\begin{enumerate}[leftmargin=0.06\textwidth]
	\item $Sq^k$ is $\mathbb F_2$-linear, $P^k$ is $\Fp$-linear.
	\item $Sq^0 = \id$ and $P^0 = \id$.
	\item $Sq^1$ is the Bockstein homomorphism.
	\item If $k > \deg x$ then $Sq^k(x)=0$.
	If $2k > \deg z$ then $P^k(z) = 0$.
	\item If $k = \deg x$, the $Sq^k(x) = x^2$.
	If $2k = \deg z$, then $P^k(z) = z^p$.
	\item (Cartan formula)
	$Sq^k(xy) = \sum_{i=0}^k Sq^i(x)Sq^{k-i}(y)$,
	$P^k(zw) = \sum_{i=0}^k P^i(z)P^{k-i}(w)$.
	\item (Adem relations)
	If $0 < l < 2k$, then
	\[
	Sq^lSq^k = \sum_{i=0}^{l/2} \binom{k-i-1}{l-2i}Sq^{k+l -i}Sq^i.
	\]
	If $0 < l < pk$, then
	\begin{align*}
	P^l P^k &= \sum_{i=0}^{l/p} (-1)^{l+i} \binom{(p-1)(k-i)-1}{l-pi}
	P^{k+l-i}P^i,
	\\
	P^l \beta P^k &=\sum_{i=0}^{l/p} (-1)^{l+i} \binom{(p-1)(k-i)}{l-pi}
	\beta P^{k+l-i}P^i
	\\
	&\phantom{=}{}+{}
	\sum_{i=0}^{(l-1)/p} (-1)^{l+i} \binom{(p-1)(k-i)-1}{l-pi}
	P^{k+l-i}\beta P^i,
	\end{align*}
	where $\beta\colon H^n(X; \Fp) \to H^{n+1}(X;\Fp)$ is the mod~$p$ Bockstein homomorphism.
\end{enumerate}

For a prime~$p$, define the mod~$p$ \emph{Steenrod algebra} $\Ap$ to be the
free $\Fp$-algebra generated by $Sq^k$, $k\ge 1$, if $p = 2$,
or by $\beta$ and $P^k$, $k\ge 1$, if $p$ is odd, modulo
the Adem relations. By the Cartan-Serre theorem, the admissible monomials  form an additive basis of $\mathfrak A_2$ (resp. $\Ap$).

The Steenrod algebra is a Hopf algebra, with a coproduct induced by the map of the Eilenberg-MacLane spectra
\[
	H\Fp \simeq \Sigma^{\infty}S^0 \wedge H\Fp \xrightarrow{\eta\wedge \id} H\Fp\wedge H\Fp. 
\]
The following theorems describes the structure of its dual Hopf algebra $\Ap^*$.

\begin{theorem}[Milnor~\cite{miln58}]\label{steenrod-dual}
	For $p = 2$,
	let $\xi_n \in (\mathfrak{A}_2^*)_{2^n-1}$ be the
	dual basis element
	\[
	\xi_n
	=
	(Sq^{2^{n-1}} Sq^{2^{n-2}}\cdots Sq^{2} Sq^1)^*.
	\]
	For odd $p$,
	let $\xi_n \in (\Ap^*)_{2(p^n-1)}$
	and
	$\tau_n \in (\Ap^*)_{2p^n-1}$
	be the dual basis elements
	\begin{align*}
	\xi_n &= (P^{p^{n-1}}P^{p^{n-2}}\cdots P^p P^1)^*,
	\\
	\tau_n &= (P^{p^{n-1}}P^{p^{n-2}}\cdots P^p P^1 \beta)^*.
	\end{align*}
	Then for $p = 2$, we have an isomorphism of algebras
	\[
	\mathfrak{A}_2^*\cong \mathbb{F}_2[\xi_1, \xi_2, \ldots],
	\]
	and for odd~$p$, we have an isomorphism of algebras
	\[
	\Ap^* \cong
	\Fp[\xi_1, \xi_2, \ldots]\otimes_{\Fp}\Lambda_{\Fp}[\tau_0, 
	\tau_1, \ldots].
	\]
	
	Let $\xi_0 = 1$. The coproduct on $\Ap^*$ is given by
	\begin{align*}
	\Delta(\xi_n) &= \sum_{k = 0}^{n} \xi_{n-k}^{p^k}\otimes\xi_k,%
	\quad\text{for all~$p$};
	\\
	\Delta(\tau_n) &= \tau_n\otimes 1 + \sum_{k = 0}^{n} 
	\xi_{n-k}^{p^k}\otimes\tau_k,%
	\quad\text{for odd~$p$}.
	\end{align*}
\end{theorem}

\begin{theorem}[Brown-Davis-Peterson~{\cite[Theorem~1.2]{b-d-p77}}]%
		\label{xi-antipode}

		Let $\xi = 1 + \xi_1 + \xi_2 + \dots$
		If $R = (r_1, r_2, \ldots)$ is a finite sequence of non-negative
		integers, let $\xi^R = \xi_1^{r_1}\xi_2^{r_2}\dots \in \Ap^*$,
		$e(R) = \sum_{i \ge 1} r_i$, $n(R) = \sum_{i \ge 1} r_i(p^i-1)$,
		and $\binom{e(R)}{r_1, r_2, \ldots}$ the multinomial coefficient.
		Then for each integer $k$,
		\[
		\bar\xi^k = \sum_{R} \binom{c(n(R) + k + 1)}{e(R)}\binom{e(R)}{r_1, r_2, \ldots}
		\xi^R,
		\]
		where $c(m) = p^i - m$, with $i$ the smallest integer
		such that $p^i - m$ is positive.
	\end{theorem}

\medskip

For an abelian group~$G$, let
\[\postdisplaypenalty=10000
p^{\infty}G = \bigcap_{r=1}^{\infty} p^rG,
\]
be the subgroup of elements divisible
by $p^r$ for any $r \geqslant 1$, and let
\[\postdisplaypenalty=10000
G_p = G/p^{\infty}G.
\]
If $G$ is finitely generated, then the latter equals $G$ modulo
torsion of order prime to~$p$.

\begin{theorem}[Adams Spectral Sequence]\label{ASS}
	Let $p$ be a prime, and let $X$ be a ring spectrum of finite type.
	Then there is a natural multiplicative spectral sequence 
	\[
	E_2^{s,t} \cong \Ext_{\Ap}^s\!\big(H^{\sbullet}(X; \Fp), \Fp\big)_t
	\Longrightarrow \pi_{t-s}(X)_p,		
	\]
	where
	\[
	d_r\colon E_r^{s,t} \to E_r^{s+r+1, t+r}.
	\]
\end{theorem}

\medskip

An element $x$ in $\Ap^*$-comodule is called \emph{primitive} if it maps to $1\otimes x$ under the structure map.

Recall that $H_{\sbullet}(MU; \Z)$is isomorphic to $\Z[b_1, b_2, \ldots]$,
where $b_i = j_*\big((c_1^i)^*\big)$ and $j\colon \C P^\infty \to MU$ is the canonical map.
The following theorem describes the $\Ap^*$-coalgebra structure of $H_{\sbullet}(MU; \Fp)$.

\begin{theorem}[{\cite[Chapter~20]{swit75}}]\label{thom-homology}
	Let $PH_{\sbullet}(MU; \Fp)$ be the subalgebra of primitive elements
	of $H_{\sbullet}(MU; \Fp)$. Then
	\[
		PH_{\sbullet}(MU; \Fp) = \Fp[x_k\,|\; k\geqslant 1, k \neq p^t - 1],\quad\deg x_k = 2k.
	\]
	Furthermore, 
	\[
	H_{\sbullet}(MU; \Fp) \cong \Ap'\otimes_{\Fp} PH_{\sbullet}(MU; \Fp),
	\]
	as $\Fp$-algebras and $\Ap^*$-comodules,
	where
	\[
		\Ap' = (\Ap/(\beta))^*\cong
		\begin{cases}
			\mathbb F_2[\xi_1^2, \xi_2^2, \ldots]\quad&\text{if $p = 2$;}
			\\
			\Fp[\xi_1, \xi_2, \ldots]&\text{if $p$ is odd.}
		\end{cases}
	\]
\end{theorem}

\section{Homology of the Thom spectrum $MSU$}\label{sunitary}
We begin with collecting the necessary information about
homology of the classifying space $BSU$.

\begin{lemma}{{\cite[Lemma~2.4]{adams75}}}\label{homology-bsu}
	There is an algebra isomorphism
	\[
		H_{\sbullet}(BSU; \Z) \cong \Z[y_2, y_3, \ldots].
	\]
	for suitable $y_i \in H_{2i}(BSU; \Z)$, $i = 2, 3, \ldots$
\end{lemma}

\begin{remark}
	\hr[Lemma]{homology-bsu} can be easily generalised. Namely,
	$E_{\sbullet}(BSU) \cong \pi_{\sbullet}(E) [Y_2, Y_3, \ldots]$ for any
	complex oriented spectrum~$E$.
\end{remark}

\begin{lemma}
	Homology $H_{\sbullet}(BSU; \Z)$ is a subalgebra of $H_{\sbullet}(BU; \Z)$.
\end{lemma}

\begin{proof}
	Consider the canonical fibration
	$BSU \xrightarrow{f} BU \xrightarrow{\det} \C P^\infty$.
	The latter map corresponds to the first Chern class
	$c_1 \in H^2(BU; \Z) \cong [BU, \C P^\infty]$.
	Since the maps in the fibration are maps of $H$-spaces,
	they induce algebra homomorphisms in homology.
	Note that
	$f_*\colon H_{\sbullet}(BSU; \Z) \to H_{\sbullet}(BU; \Z)$
	is monic. Therefore,
	$H_{\sbullet}(BSU; \Z) \cong \im f_*$
	is a subalgebra of $H_{\sbullet}(BU; \Z)$.	
\end{proof}

\begin{lemma}\label{p_th}
	The $p$-th power of any element $H_{\sbullet} (BU; \Fp)$
	lies in $H_{\sbullet}(BSU; \Fp)$.
\end{lemma}

\begin{proof}
	Consider cohomology $H^{\sbullet}(BU;\Z)$ as a $\Z[c_1]$-module.
	Dualising the action map
	$\Z[c_1]\otimes_{\Z} H^{\sbullet}(BU;\Z) \to H^{\sbullet}(BU; \Z)$
	we obtain a coaction
	\[
	\psi\colon H_{\sbullet}(BU; \Z)
	\to
	\Gamma_\Z[\beta_1]\otimes_\Z H_{\sbullet}(BU; \Z),
	\]
	where $\Gamma_\Z[\beta_1]$ is a divided polynomial algebra.
	The subcoalgebra $P_\psi H_{\sbullet}(BU; \Z)$ of primitive elements
	is isomorphic to $\im f_* \cong H_{\sbullet}(BSU; \Z)$.
	In particular, since for any non-constant
	$x \in \Gamma_\Z[\beta_1]\otimes_\Z \Fp$
	its $p$-th power is zero, the $p$-th power of any element of
	$H_{\sbullet}(BU; \Fp)$ actually lies in $H_{\sbullet}(BSU; \Fp)$.	
\end{proof}

The Thom isomorphism
$\Phi\colon H_{\sbullet}(BSU; \Z) \to H_{\sbullet}(MSU; \Z)$
is an algebra isomorphism, which implies
\[
	H_{\sbullet}(MSU; \Z) \cong \Z[Y_2, Y_3, \ldots].
\]
It follows that the $p$-th power of any element of $H_{\sbullet}(MU; \Fp)$
lies in $H_{\sbullet}(MSU; \Fp)$.

The following theorem describes specific polynomial generators
of $H_{\sbullet}(MU; \Fp)$ which are compatible with
the inclusion $f_*\colon H_{\sbullet}(MSU; \Fp) \to H_{\sbullet}(MU; \Fp)$.

\begin{theorem}[cf.~\cite{peng82}]\label{preMSU-coalgebra}
	Let $p$ be an odd prime.
	There are elements $z_n \in H_{2n}(MU; \Fp)$, $n \ge 1$, such that the following hold
	\begin{itemize}[leftmargin=0.06\textwidth]
		\item[\textup(i\textup)]
		$H_{\sbullet}(MU; \Fp) \cong \Fp[z_1, z_2, \ldots]$\textup;
		\item[\textup(ii\textup)]
		the composite
		\[
		G\colon H_{\sbullet}(MU;\Fp)
		\xrightarrow{\rho}
		\Ap'\otimes_{\Fp}H_{\sbullet}(MU;\Fp)
		\xrightarrow{\id\otimes \pi}
		\Ap'\otimes_{\Fp}H_{\sbullet}(MU;\Fp)/(z_i, i = p^t - 1),
		\]
		where $\rho$ is the left coaction map and $\pi$ is the canonical projection,
		is an isomorphism of $\Fp$-algebras and $\Ap^*$-comodules.
		Here the structure of an $\Ap^*$-comodule of the latter algebra
		is given by the coaction on the first tensor factor\textup;
		\item[\textup(iii\textup)]
		$G(z_{p^t-1}) = - \bar\xi_t\otimes 1$, $t \ge 1$, where  $\bar\xi_t$ is the Hopf conjugate of $\xi_t$, and
		$G(z_n) = 1\otimes z_n$, $n \neq p^t - 1$\textup;
		\item[\textup(iv\textup)]
		if $Y_n \in H_{2n}(MU; \Fp)$, $n \ge 2$, are defined by
		\[
		Y_n =
		\begin{cases}
		z_{n/p}^p\quad&\text{if $n = p^t$};
		\\					
		z_n&\text{otherwise};
		\end{cases}
		\]
		then
		$H_{\sbullet}(MSU;\Fp)
		{}\cong{}
		\Fp[Y_2, Y_3, \ldots] \subset H_{\sbullet}(MU; \Fp)$.
	\end{itemize}
\end{theorem}

\begin{proof}
	If $n \neq p^t, p^t - 1$, let $z_n$ be the Hurewicz image of 
	the $n$-dimensional $SU$-manifold $M_n$, where $\{M_n\}_{n\neq p^t, p^t-1}$
	are polynomial generators of $\pi_{\sbullet}(MU)\otimes_{\Z}\Fp$ in
	the given degrees. Such manifolds are described, for example,
	in~\cite[p.~240--242]{ston68}.
	Note that $z_n$ is primitive as it is the Hurewicz image.
	
	If $n = p^t$, take $z_n = x_{p^t} \in H_{2p^t}(MU; \Fp)$, where $x_{p^t}$ is from \hr[Theorem]{thom-homology}.
	Elements $z_n$ are primitive by the very definition.
	
	Finally, if $n = p^t-1$, let $z_{p^t-1} = \Phi(c_{p^t-1}^*)$, $t\ge 1$,
	where $c_{p^t-1}^* \in H_{\sbullet}(BU; \Fp)$ is the dual of $c_{p^t-1}$
	with respect to the monomial basis in the $c_i$.
	Clearly, $z_{p^t - 1} \in H_{\sbullet}(MSU; \Fp)$ for all $t \ge 1$.
	To compute the coaction of $z_{p^t-1}$ we use the following theorem.
	
	\begin{theorem}[Brown-Davis-Peterson~{\cite[Theorem~1.1]{b-d-p77}}]\label{chern-coaction}
		Let $\xi = 1 + \xi_1 + \xi_2 + \dots$ and
		$C^* = 1 + \Phi(c_1^*) + \Phi(c_2^*) + \dots$
		Then for the right action
		$H^{\sbullet}(MU; \Fp) \otimes_{\Fp}\Ap \to H^{\sbullet}(MU; \Fp)$,
		the dual right coaction $\Delta$ maps $C^*$ to
		\[
		\Delta(C^*) = -1 \otimes \bar\xi^{-1} +{}
		\sum_{i \ge 1} \Phi(c_i^*)\otimes\bar\xi^{i-1}.
		\]
		Here and further $\bar\psi \in \Ap^*$ is Hopf conjugate of $\psi \in \Ap^*$,
	\end{theorem}
		
	\begin{corollary}\label{411}
		$\Delta(z_{p^t-1}) = -1\otimes\xi_t + \sum\limits_{s = 0}^{t-1}
		z_{p^{t-s}-1}^{p^{s}} \otimes \xi_{s}$.
		Equivalently, $\rho(z_{p^t-1}) = - \bar\xi_t\otimes 1 + \sum\limits_{s 
			= 0}^{t-1}
		\bar\xi_{s} \otimes z_{p^{t-s}-1}^{p^{s}}$.
	\end{corollary}
	\begin{proof}
		For $\zeta\in \Ap^*$, denote by
		$\zeta_n$ the component of degree~$n$.
		
		Since $\Phi(c_{p^t-1}^*) = z_{p^t-1}$, the formula from \hr[Theorem]{chern-coaction}
		gives
		\begin{align}\label{cor37}
			\Delta(z_{p^t-1}) = -1\otimes (\bar\xi^{-1})_{2(p^t-1)} + \sum_{i \geqslant 1} \Phi(c_i^*)\otimes (\bar\xi^{i-1})_{2(p^t-i-1)}.
		\end{align}
		
		First, it follows easily from \hr[Theorem]{xi-antipode} that
		$(\bar\xi^{-1})_{2(p^t - 1)} = \xi_t$.

		By \hr[Theorem]{xi-antipode},
		\begin{align*}
			(\bar\xi^{i-1})_{2(p^t-i-1)}
			&=
			\sum_{n(R) = p^t - i - 1}
			\binom{c((p^t - i - 1) + (i-1) + 1)}{e(R)}\binom{e(R)}{r_1, r_2, \ldots}\xi^R
			\\
			&=
			\sum_{n(R) = p^t - i - 1}
			\binom{1}{e(R)}\binom{e(R)}{r_1, r_2, \ldots}\xi^R
			=
			\begin{cases}
				0\quad&\text{if $i \neq p^t- p^s$;}
				\\
				\xi_{s}&\text{if $i = p^t-p^s$.}
			\end{cases}
		\end{align*}
		The last identity holds because $\binom{1}{e(R)}=0$ unless $R = (0, \ldots, 1, 0, \ldots)$, in which case $n(R) = p^s-1$, $s \geqslant 1$.
		Finally, $\Phi(c_{pk}^*) = \Phi(c_k^*)^p$, hence $\Phi(c_{p^t - p^s}) = z_{p^s-1}^{p^{t-s}}$.
		Substituting these expressions into \eqref{cor37},
		we obtain
		\[
			\Delta(z_{p^t-1}) = -1\otimes\xi_t + \sum\limits_{s = 0}^{t-1}
			z_{p^{t-s}-1}^{p^{s}} \otimes \xi_{s}
		\]
		as needed.
	\end{proof}
	
	\medskip
	
	We resume the proof of \hr[Theorem]{preMSU-coalgebra}.
	Since $\rho$ is multiplicative, it follows from the corollary above that 
	$z_{p^t-1}$ is indecomposable in $H_{\sbullet}(MU; \Fp)$. The elements $z_i$, $i \neq p^t-1$, are indecomposable in $H_{\sbullet}(MU;\Fp)$ by their construction. This proves assertion~(i). 
	Assertions (ii)--(iii) are clear.
	If $n \neq p^t$, then $z_n$ actually lies in the image of $H_{\sbullet}(MSU; \Fp)$.
	Therefore, they are polynomial generators of $H_{\sbullet}(MSU; \Fp)$.
	By~\hr[Lemma]{p_th}, $z_{p^{t-1}}^p$, $t \geqslant 1$, lies in $H_{\sbullet}(MSU; \Fp)$. And finally, by~\cite[Lemma~2.1]{adams75}, $z_{p^{t-1}}^p \in H_{p^t}(MSU; \Fp)$, $t \geqslant 1$ are polynomial generators.
\end{proof}

\begin{corollary}\label{38}
    There are polynomial generators $Y_n \in H_{2n}(MSU)$, $\neq p^t$ such that $f_*(Y_n) = \pm b_n$ modulo decomposables.
\end{corollary}

\medskip

We summarise the results above in the following description of the $\Ap^*$-coalgebra structure of~$H_{\sbullet}(MSU; \Fp)$, it is similar to \hr[Theorem]{thom-homology}.

\begin{theorem}\label{MSU-coalgebra}
	Let $p$ be an odd prime.
	There are elements $Y_n \in H_{2n}(MSU; \Fp)$ such that
	\[
		H_{\sbullet}(MSU; \Fp) \cong \Fp[Y_2, Y_3, \ldots].
	\]
	The left coaction $\rho\colon H_{\sbullet}(MSU; \Fp) \to \mathfrak A_p^* \otimes_{\Fp} H_{\sbullet}(MSU; \Fp)$ is given by
	\[
		Y_n \mapsto
		\begin{cases}
			- \bar\xi_t\otimes 1 + \sum\limits_{s 
				= 0}^{t-1}
			\bar\xi_{s} \otimes Y_{p^{t-s}-1}^{p^{s}}\quad&\text{if $n = p^t - 1$;}
			\\
			\hfil Y_n&\text{otherwise.}
		\end{cases} 
	\]
	In particular, the subalgebra $PH_{\sbullet}(MSU; \Fp)$ of primitive elements under this coaction is isomorphic to
	\[
		\Fp[Y_n, n\neq p^t - 1, n \geqslant 2],\quad\deg Y_n = 2n.
	\]
	Futhermore, the composite 
	\begin{align*}
		H_{\sbullet}(MSU; \Fp)
		\xrightarrow{\rho}&
		\Ap' \otimes_{\Fp} H_{\sbullet}(MSU; \Fp)
		\\
		\xrightarrow{\id\otimes \pi}&
		\Ap'\otimes_{\Fp} H_{\sbullet}(MSU; \Fp)/(Y_i, i = p^t-1)
		= \Ap'\otimes_{\Fp} PH_{\sbullet}(MSU; \Fp)
	\end{align*}%
	is an isomorphism of $\Fp$-algebras and $\Ap^*$-comodules. Moreover, this composite is given by
	\[
		Y_n \mapsto
		\begin{cases}
			-\bar\xi_t \otimes 1\quad&\text{if $n = p^t-1$;}
			\\
			1\otimes Y_n&\text{otherwise}.
		\end{cases}
	\]
\end{theorem}

\section{Novikov's theorem}\label{novikov_thm}
In this section we compute $\pi_{\sbullet}(MSU)\otimes \Z[\frac 12]$ using the modern adaptation of
the Novikov's original proof (see~\cite{novi62}).

\begin{theorem}[Novikov~\cite{novi62}]\label{pi-msu}
	There are elements $y_n \in \pi_{2n}(MSU)\otimes \Z[\frac 12]$,
	$n = 2, 3, \ldots$, such that
	\[
		\pi_{\sbullet}(MSU)\otimes \Z[\textstyle\frac{1}{2}]
		\cong
		\Z[\textstyle\frac{1}{2}][y_2, y_3, \ldots].
	\]
\end{theorem}

\begin{proof}
	For an odd prime~$p$, consider the mod~$p$ Adams spectral sequence for
	$\pi_{\sbullet}(MSU)$ with the second term
	\[
		E_2^{s,t} \cong \Cotor_{\Ap^*}^s\big(\Fp, H_{\sbullet}(MSU; \Fp)\big)_t.
	\]
	
	By~\hr[Theorem]{MSU-coalgebra}, $H_{\sbullet}(MSU; \Fp) \cong \Ap'\otimes PH_{\sbullet}(MSU; \Fp)$,
	where $\Ap' = \Fp[\xi_1, \xi_2, \ldots]$ and $PH_{\sbullet}(MSU; \Fp) = \Fp[Y_n, n\neq p^t - 1, n \geqslant 2], \deg Y_n = 2n$,
	is the subcoalgebra of primitive elements.
	
	By The Change of Rings Theorem (see~\cite{liul64}),
	\begin{align*}
		\Cotor_{\Ap^*}^{\sbullet}\big(\Fp, H_{\sbullet}(MSU;\Fp)\big)_{\sbullet}
		&=
		\Cotor_{\Ap^*}^{\sbullet}
		\big(\Fp, \Ap'\otimes_{\Fp} PH_{\sbullet}(MSU; \Fp)\big)_{\sbullet}
		\\
		&=
		\Cotor_{\Ap^*}^{\sbullet}
		\big(\Fp, \Ap'\big)_{\sbullet}\otimes_{\Fp} PH_{\sbullet}(MSU; \Fp)
		\\
		&=
		\Cotor_{\Ap^*/\!/\Ap'}^{\sbullet}
		(\Fp, \Fp)_{\sbullet}\otimes_{\Fp} PH_{\sbullet}(MSU; \Fp).	
	\end{align*}
	
	By \hr[Theorem]{steenrod-dual}, $\Ap^*/\!/\Ap' \overset{\mathrm{def}}{=} \Ap^*/(\Ap^*\cdot\Ap') = \Lambda_{\Fp}[\tau_0, \tau_1, \ldots]$, $\deg \tau_t = 2p^t-1$,
	is an exterior algebra of finite type. Therefore, 	$\Cotor_{\Ap^*/\!/\Ap'}^{\sbullet}(\Fp, \Fp)_{\sbullet}$ is a polynomial algebra
	\[
		\Cotor_{\Ap^*/\!/\Ap'}^{\sbullet}(\Fp, \Fp)_{\sbullet} = \Fp[q_0, q_1, \ldots],\quad q_t \in \Cotor_{\Ap^*/\!/\Ap'}^1(\Fp, \Fp)_{2p^t-1}.
	\]
	Therefore, for odd prime~$p$ the second term of the Adams spectral sequence has the following form
	\begin{align*}
		E_2^{\sbullet, \sbullet}
		&=
		\Fp[q_0, q_1, \ldots]\otimes_{\Fp} PH_{\sbullet}(MSU;\Fp)
		\\
		&=
		\Fp[q_0, q_1, \ldots]\otimes_{\Fp} \Fp[Y_n, n\neq p^t - 1, n \geqslant 2]
		\\
		&=
		\Fp[q_0] \otimes_{\Fp} \Fp[m_2, m_3, \ldots],
	\end{align*}
	where $q_0 \in E_2^{1,1}$ and
	\[
		m_n =
		\begin{cases}
		q_t\in E_2^{1, 2n+1}\quad&\text{if $n = p^t-1$ for some $t>0$;}
		\\
		Y_n \in E_2^{0, 2n}&\text{otherwise}.
		\end{cases}
	\]
	Note that $E_2^{s,t}$ is zero whenever $t-s$ is odd.
	Therefore, since the differentials reduce $t-s$ by~$1$,
	$d_r = 0$ for every $r = 2, 3, \ldots$ and $E_{2}^{\sbullet, \sbullet} = E_{\infty}^{\sbullet, \sbullet}$.
	One can easily show that
	\[
		q_0 \in E_2^{1,1} = \Cotor_{\Ap^*/\!/\Ap'}^1(\Fp, \Fp)_1
	\]
	is represented by $p \in \pi_0(MSU) \cong \Z$.
	Thus, for every odd prime~$p$ there is an isomorphism of abelian groups
	\[
		\pi_{\sbullet}(MSU)_p \cong \Z[m_2, m_3, \ldots].
	\]
	
	Next we show that there is a ring isomorphism
	\[
		\pi_{\sbullet}(MSU)\otimes_{\Z} \Fp \cong \Fp[P_2, P_3, \ldots]
	\]
	for some $P_i = P_i(p)\in \pi_{2i}(MSU)$.
		
	Let $P_i\in \pi_{2i}(MSU)$, $i = 2, 3, \ldots$ be elements corresponding
	to $m_i \in E_\infty^{\sbullet, \sbullet}$ in the mod~$p$ Adams spectral 
	sequence for $MSU$. Note that our choice depends on~$p$.
	Since $m^{\alpha}$,
	$\alpha = \{\alpha_1, \ldots, \alpha_k, 0, \ldots\} \in \bigoplus_{i = 0}^\infty \Z$
	are linearly independent in the Adams spectral sequence, it follows that the monomials
	$P^\alpha = P_1^{\alpha_1}\cdots P_k^{\alpha_k}$ are linearly independent.
	We have to prove that the monomials $P^\alpha$ generate
	$\pi_{\sbullet}(MSU)\otimes_\Z \Fp$.
	
	\cite[Lemma~20.28]{swit75} states that for every~$q \in \Z_{\geqslant 0}$
	there is $t\in \Z_{\geqslant 0}$
	such that a filtration term $F^{t,q+t}$ is trivial $\mod p$,
	i.e. $F^{t, q+t} \otimes_\Z \Fp = 0$.
	Now using downward induction over $s$ starting at $s = t$, we show that
	monomials $P^\alpha$ of degree~$q$ and of filtration at least~$s$ span 
	$F^{s,q+s}\otimes_\Z \Fp$.
	Let $\{m^\alpha, \alpha\in\bigoplus_{i = 0}^\infty \Z\}$
	be monomials in $m_i$ in $E_{\infty}^{s-1, q+s-1}$, then any $x \in F^{s-1, q+s-1}\otimes_\Z\Fp$
	can be written as $x = \sum a_i P^{\alpha_i} \mathrm{mod}\,F^{s,q+s}$, 
	because the projection $F^{s-1,q+s-1} \to E_{\infty}^{s-1, q+s-1}$
	sends $P^{\alpha_i}$ to $m^{\alpha_i}$, $i = 1, \ldots, k$. 
	Now we can use the inductive assumption to prove that
	the monomials $P^\alpha$ generate $\pi_{\sbullet}(MSU)\otimes_\Z \Fp$.
	
	Finally, \hr[Theorem]{pi-msu} follows from the fact that if a graded ring of finite type~$R$ is isomorphic to $\Fp[P_2, P_3, \ldots]$, $\deg P_i = 2i$ modulo an odd prime~$p$ then its localisation $R[\frac 12]$ is isomorphic to $\Z[\frac 12][y_2, y_3, \ldots]$, $\deg y_i = 2i$.
\end{proof}

\section{Multiplicative generators}\label{generators}
In this section we compute characteristic numbers of
polynomial generators of $\pi_{\sbullet}(MSU)\otimes \Z[\frac 12]$.
We start from the description of the mod~$p$ Hurewicz
homomorphism in terms of the mod~$p$ Adams spectral sequence.

\begin{proposition}\label{hurewicz-ASS}
	Let $X$ be a spectrum of finite type. Then the mod~$p$ Hurewicz homomorphism
	\[
	h\colon \pi_{\sbullet}(X) \to H_{\sbullet}(X; \Fp)
	\]
	coincides with the composite map
	\[
	\pi_{\sbullet}(X) \to E_\infty^{0,{\sbullet}} \hookrightarrow 
	E_2^{0,{\sbullet}} = PH_{\sbullet}(X; \Fp) 
	\hookrightarrow H_{\sbullet}(X; \Fp),
	\]
	where
	$E_\infty^{0,{\sbullet}}$ and $E_2^{0,{\sbullet}}$
	are terms of the mod~$p$ Adams spectral sequence.
\end{proposition}
\begin{proof}
	Note that the map
	\[
	h'\colon X \simeq \Sigma^\infty S^0 \wedge X
	\xrightarrow{\eta_{H\Fp}\wedge \id_X} H\Fp\wedge X
	\]
	induces the mod~$p$ Hurewicz homomorphism in homotopy groups.
	Thus, we have a map of mod~$p$ Adams spectral sequences
	$(h'_*)_r \colon E_r^{\sbullet, \sbullet}
	\to
	\bar{E}_r^{\sbullet, \sbullet}$, $r \geqslant 2$,
	for $\pi_{\sbullet}(X)$
	and $\pi_{\sbullet}(H\Fp\wedge X) \cong H_{\sbullet}(X; \Fp)$
	induced by~$h'$. 
	Since
	$H_{\sbullet}(H\Fp\wedge X; \Fp)
	\cong \Ap^*\otimes_{\Fp} H_{\sbullet}(X; \Fp)$,
	it follows that 
	\[
	\bar{E}_2^{s, \sbullet} =
	\Cotor_{\Ap^*}^s\big(\Fp, H_{\sbullet}(H\Fp\wedge X; 
	\Fp)\big)_{\sbullet} = 0,\quad\text{if $s>0$}
	\]
	and
	\[
	\bar{E}_2^{0, \sbullet}
	=
	\Cotor_{\Ap^*}^s\big(\Fp, H_{\sbullet}(H\Fp\wedge X; 
	\Fp)\big)_{\sbullet}
	\cong
	\Fp\square_{\Ap^*} H_{\sbullet}(H\Fp\wedge X; \Fp)
	=
	H_{\sbullet}(X; \Fp).
	\]
	The proposition now follows from the following decomposition
	of the Hurewicz homomorphism
	\begin{equation*}
	\begin{tikzcd}
	\pi_{\sbullet}(X)
	\arrow{r}{h}
	\arrow{d}
	&
	H_{\sbullet}(X; \Fp)
	\arrow{d}{\cong}
	\\
	E_\infty^{0, \sbullet}
	\arrow{r}{(h'_*)_\infty}
	\arrow[hook]{d}
	&
	\bar{E}_\infty^{0, \sbullet}
	\arrow[equal]{d}
	\\
	E_2^{0, \sbullet}
	\arrow{r}{(h'_*)_2}
	&
	\bar{E}_2^{0, \sbullet}
	\end{tikzcd}
	\end{equation*}
	and the fact that $(h'_*)_2$ equals to the following composition
	$E_2^{0, \sbullet} \cong PH_{\sbullet}(X; \Fp)
	\hookrightarrow H_{\sbullet}(X; \Fp)$.
\end{proof}

\medskip

For any positive integer~$n$ define
\[
		\lambda_n =
		\begin{cases}
			p\quad&\text{if $n+1 = p^t$ for some positive~$t$ and prime~$p$;}
			\\
			1&\text{otherwise.}
		\end{cases}
\]

\begin{theorem}[Novikov~\cite{novi62}]\label{polynomial-generators}
Let $h\colon \pi_{\sbullet}(MSU)\otimes[\frac 12] \to H_{\sbullet}(MSU;\Z[\frac12])$ be the Hurewicz homomorphism and let $f\colon MSU\to MU$ be the canonical map. Then there exist polynomial generators $y_n\in \pi_{2n}(MSU)\otimes \Z[\frac 12]$, $n\ge2$, such that
\[
  f_*h(y_n) = \pm \lambda_n\lambda_{n-1} b_n
\] 
modulo decomposable elements in $H_{2n}(MU;\Z[\frac12])$, where $b_n\in H_{2n}(MU)$ is a canonical polynomial generator.
	
Therefore, $y_n\in \pi_{2n}(MSU)\otimes \Z[\frac 12]$ can be taken as a polynomial generator if and only if
	\[
		s_n(y_n) = \pm \lambda_n \lambda_{n-1}
	\]
	up to a power of 2.
\end{theorem}

\begin{proof}
	\emph{Case 1:} $n\neq p^t-1$ for an odd prime~$p$.
	By \hr[Proposition]{hurewicz-ASS},
	$y_n \in \pi_{\sbullet}(MSU)\otimes_{\Z}\Z[\frac12]$ is a polynomial
	generator if and only if
	$h(y_n) \in PH_{\sbullet}(MSU; \Fp)\cong \Fp[Y_n\,|\; n \neq p^t - 
	1,\text{and}~n\geqslant 2]$ is a polynomial generator,
	in particular it is a polynomial generator of
	$H_{\sbullet}(MSU; \Fp)$. Since this is true for every odd prime~$p$,
	$h(y_n)$ is a polynomial generator of $H_{\sbullet}(MSU; \Z[\frac12])$.
	Thus, $h(y_n) = \pm Y_n$ modulo decomposables.
	
	Suppose that $n\ne p^t$. 
	Then $f_*(Y_n)=\pm b_n$ modulo decomposables by \hr[Corollary]{38}. Hence, $f_*h(y_n) = \pm b_n$ modulo decomposables, as needed. 
	
	Now suppose $n = p^t$.
    Then $f_*(Y_n)=\pm pb_n$ modulo decomposables by~\cite[Theorem~3.3]{koch82}. Therefore, $f_*h(y_n) = \pm p b_n=\pm\lambda_n\lambda_{n-1}b_n$ modulo decomposables, as claimed.
	
	\smallskip
	
	\emph{Case 2:} $n = p^t-1$ for an odd prime~$p$.
	Let $y_n \in \pi_{2n}(MSU)\otimes \Z[\frac 12]$ be a polynomial generator.
	Then $h(y_n)$ is decomposable in
	$PH_{\sbullet}(MSU; \Fp) \cong \Fp[Y_k\,|\; k \neq p^t - 
	1,\text{and}~k\geqslant 2]$
	since there are no indecomposables in degree $2(p^t - 1)$.
	Therefore, $f_*h(y_n) \in H_{2n}(MU; \Z)$ is
	divisible by~$p$ modulo decomposables.
	Argument similar to the one in case $n \neq p^t-1$ implies
	that $h(y_n)$ is a polynomial generator of
	$PH_{\sbullet}(MSU; \mathbb{F}_q)$ for any prime~$q\neq p$.
	Hence, $f_*h(y_n)$ is divisible by~$p$
	and is not divisible by any other prime~$q$ modulo decomposables.
	Therefore, $f_*h(y_n) = \pm p^k b_n$ modulo decomposables.
	
	It remains to prove that $k=1$ above. In other words, we need to prove that $h(y_n)$
	is not divisible by~$p^2$  modulo decomposables.
	This can be done either algebraically as below, or by providing
	explicit examples of $SU$-manifolds with an appropriate Milnor genus (see~\cite[Part~\RN{2}]{c-l-p19}).
	
	Consider the map $h'\colon MSU 
	\xrightarrow{\eta_{H\Z}\wedge\id_{MSU}} H\Z\wedge MSU$
	that induces the Hurewicz homomorphism in homotopy groups.	The second term of the mod~$p$ Adams spectral sequence for $H\Z\wedge MSU$
	is
	\begin{equation*}
    	\bar E_2^{s,t} =
    	\Cotor_{\Ap^*}^{s}(\Fp; H_{\sbullet}(H\Z\wedge MSU; \Fp))_t
    	=
    	\Cotor_{\Ap^*}^s(\Fp; H_{\sbullet}(H\Z; \Fp))_{\sbullet} 
    	\otimes_{\Fp} H_{\sbullet}(MSU; \Fp).
	\end{equation*}

	We have $H_{\sbullet}(H\Z; \Fp) \cong \Fp[\xi_1, \xi_2, \ldots]\otimes_{\Fp}\Lambda_{\Fp}[\bar\tau_1, \bar\tau_2, \ldots]$. Also, $\Ap^* \cong
	\Fp[\xi_1, \xi_2, \ldots]\otimes_{\Fp}\Lambda_{\Fp}[\bar\tau_0, 
	\bar\tau_1,\bar\tau_2, \ldots]$ by \hr[Theorem]{steenrod-dual}. It follows that $\Cotor_{\Ap^*}^{\sbullet}(\Fp; H_{\sbullet}(H\Z; \Fp))_{\sbullet} \cong \Fp[q_0]$, where $q_0 \in \Cotor_{\Ap^*}^1(\Fp; H_{\sbullet}(H\Z; \Fp))_1$ represents $p \in \pi_0(H\Z)$.
		
	Therefore, $\bar E_2^{s,t} = 0$ when $t-s$ is odd.
	As before, this implies that the spectral sequence degenerates in
	the second term, i.e. $\bar E_2^{s,t} = \bar E_\infty^{s,t}$.
	
	The map $h'$ induces the following map of $\Cotor$:
\begin{equation}\label{hcotor}
	    h'_*\colon\Cotor_{\Ap^*}^{s}(\Fp; H_{\sbullet}(MSU; \Fp))_t
	    \to \Cotor_{\Ap^*}^{s}(\Fp; H_{\sbullet}(H\Z\wedge MSU; \Fp))_t \cong \Fp[q_0] \otimes_{\Fp} H_{\sbullet}(MSU; \Fp).
\end{equation}	    
    In order to describe this map, we consider the following diagram of cobar complexes:
	\begin{equation}\label{cobar-triangle}
    	\begin{tikzcd}
    		C_{\Ap^*}^{\sbullet}\big(H_{\sbullet}(MSU; \Fp)\big)
    		\arrow{r}{\sim_{\text{quasi}}}
    		\arrow{rd}{\widetilde{h}}
    		&
    		C_{\Ap^*/\!/\Ap'}^{\sbullet}\big(PH_{\sbullet}(MSU; \Fp)\big)
    		\\
    		&
            C_{\Lambda_{\Fp}[\tau_0]}^{\sbullet} \big(H_{\sbullet}(MSU; \Fp)\big)
        \end{tikzcd}
	\end{equation}	
    In homology, the map $\widetilde h$ induces the map $h'_*$ in~\eqref{hcotor}.	

Consider the class $m_n = q_t \in \Cotor_{\Ap^*}^1
\big(\Fp, H_{\sbullet}(MSU; \Fp)\big)_{2p^t-1}$
defined in the proof of \hr[Theorem]{pi-msu}.	
By Cobar construction, it is represented by
$[\bar\tau_t] \in C_{\Ap^*/\!/\Ap'}^{\sbullet}\big(PH_{\sbullet}(MSU; \Fp)\big)$.
	The element
	\[
	Q_t
	=
	- \sum_{i = 0}^{t} [\bar\tau_i]\bar\xi_{t-i}^{p^i}\in
	C_{\Ap^*}^{\sbullet}\big(H_{\sbullet}(MSU; \Fp)\big)
	\]
	is a cycle that maps to $-[\bar\tau_t]$ under the 
	horizontal arrow in \hyperref[cobar-triangle]{diagram~\eqref{cobar-triangle}}.
	Indeed, recall that the conjugation is a coalgebra antihomomorphism
	and the coproduct on $\Ap$ is given in \hr[Theorem]{steenrod-dual}.
	Therefore, we have 
	\begin{align*}
	dQ_t
	=
	-\sum_{i = 0}^{t}[1|\bar\tau_i]\bar\xi_{t-i}^{p^i}
	+\sum_{i = 0}^{t}[1|\bar\tau_i]\bar\xi_{t-i}^{p^i}
	&+\sum_{i = 0}^{t}\sum_{k = 0}^{i}%
	[\bar\tau_k|\bar\xi_{i-k}^{p^k}]\bar\xi_{t-i}^{p^i}
	-\sum_{i = 0}^{t}\sum_{k = 0}^{t-i}%
	[\bar\tau_i|\bar\xi_k^{p^i}]\bar\xi_{t-i-k}^{p^{i+k}}
	\\
	&=
	\sum\limits_{\substack{a+b = c \\ 0 \leqslant a,b,c \leqslant t}}
	[\bar\tau_a|\bar\xi_b^{p^a}]\xi_{t-c}^{p^c}
	-\hspace*{-1.5mm}
	\sum\limits_{\substack{a+b = c \\ 0 \leqslant a,b,c \leqslant t}}
	[\bar\tau_a|\bar\xi_b^{p^a}]\xi_{t-c}^{p^c}
	=
	0.
	\end{align*}
	The horizontal arrow in \hyperref[cobar-triangle]{diagram~\eqref{cobar-triangle}} maps all~$\xi_i$'s
	except $\xi_0 = 1$ to zero.	Therefore, $Q_t$ represents
	$q_t \in \Cotor_{\Ap^*}^1\big(\Fp, H_{\sbullet}(MU; \Fp)\big)_{2p^t-1}$.
	
	Now we have $\widetilde{h}(Q_t) = [\tau_0]\bar\xi_t \in C^1_{\Lambda[\tau_0]}\big(H_{\sbullet}(MSU; \Fp)\big)$.
	Hence, $(h')_*(q_t) = q_0\bar\xi_t \in \bar E_2^{1, 2p^t-1}$.
	Since $q_0$ represents the multiplication by $p$, the element $h(y_n)\in H_{2n}(MSU;\Z[\frac12])$ is divisible precisely by~$p$ modulo decomposables.
\end{proof}

\begin{corollary}\label{comm_sq}
	There are polynomial generators
	$y_i \in \pi_{2i}(MSU)\otimes_\Z \Z[\frac 12]$,
	$x_i \in \pi_{2i}(MU)\otimes_\Z \Z[\frac 12]$,
	$Y_i \in H_{2i}(MSU; \Z[\frac 12])$
	and
	$X_i \in H_{2i}(MU; \Z[\frac 12])$, such that the canonical maps
	are given as in the diagram below:
	\begin{equation*}
		\begin{tikzcd}[column sep=8em, row sep=4em]
			\pi_{\sbullet}(MSU)\otimes_\Z \Z[\textstyle{\frac 12}]
			\arrow{r}{y_n \mapsto \lambda_{n-1}x_n}
			\arrow[d, "y_n \mapsto \lambda_{n}Y_n" labl]
			&
			\pi_{\sbullet}(MU)\otimes_\Z \Z[\textstyle{\frac 12}]
			\arrow[d, "x_n \mapsto \lambda_n X_n" labl]
			\\
			H_{\sbullet}(MSU; \Z[\frac 12])
			\arrow{r}{Y_n \mapsto \lambda_{n-1} X_n}
			&
			H_{\sbullet}(MU; \Z[\frac 12])
		\end{tikzcd}
	\end{equation*}
\end{corollary}

\end{document}